\documentclass{article}
\usepackage{amsmath,amssymb,amsfonts,amsthm}
\usepackage{mathtools,mathrsfs}
\usepackage{color}
\usepackage{hyperref}
\usepackage{cleveref}
\usepackage[a4paper]{geometry}
\usepackage{algorithm}
\usepackage{algpseudocode}
\usepackage{stmaryrd}
\usepackage{subcaption}
\usepackage{xspace}

\usepackage{tikz}
\usepackage{pgfplots}
\usepackage{pgfplotstable}
\pgfplotsset{width=.7\textwidth}

\newcommand{\norm}[1]{\lVert #1 \rVert}

\newcommand{\C}{\ensuremath{\mathbb{C}}}
\newcommand{\W}{\ensuremath{\mathbb{W}}}
\newcommand{\rat}{\mathcal{Q}}

\newcommand{\pol}{\ensuremath{\mathbb{P}}}
\newcommand{\ttTBRK}{{\tt TT-TBRK}\xspace}
\newcommand{\tuckTBRK}{{\tt Tuck-TBRK}\xspace}

\newcommand{\tuck}[1]{\ensuremath{\llbracket #1 \rrbracket}}
\newcommand{\ttrain}[1]{\ensuremath{\{ #1 \}}}

\providecommand{\keywords}[1]
{	
  \textbf{{Keywords:}} #1
}

\newtheorem{lemma}{Lemma}[section]
\newtheorem{theorem}[lemma]{Theorem}
\newtheorem{corollary}[lemma]{Corollary}
\newtheorem{proposition}[lemma]{Proposition}

\theoremstyle{remark}
\newtheorem{remark}[lemma]{Remark}

\theoremstyle{definition}
\newtheorem{definition}[lemma]{Definition}

\newcommand{\tbrk}{tensorized block rational Krylov }

\renewcommand{\vec}[1]{\boldsymbol{#1}}		

\author{
    Angelo A. Casulli\thanks{
        Scuola Normale Superiore, Pisa, Italy
        (\texttt{angelo.casulli@sns.it}).}
        \thanks{The author is a member of the research group INdAM-GNCS.}
}
\title{Tensorized block rational Krylov methods for tensor Sylvester equations}
\date{}

\begin{document}
    \maketitle 

    \begin{abstract}
        We introduce the definition of tensorized block rational Krylov subspaces and its relation with multivariate rational functions, extending the formulation of tensorized Krylov subspaces introduced in [Kressner D., Tobler C., \emph{{K}rylov subspace methods for linear systems with
        tensor product structure}, SIMAX, 2010]. Moreover, we develop methods for the solution of tensor Sylvester equations with low multilinear or Tensor Train rank, based on projection onto a tensor block rational Krylov subspace. We provide a convergence analysis, some strategies for pole selection, and techniques to efficiently compute the residual.
     
    \end{abstract}


    \section{Introduction}
    In this work, we develop methods based on projection onto block rational Krylov subspaces, for solving tensor Sylvester equations 
    \begin{equation}{\label{eqn:tensor_sylv}}
        \mathcal{X}\times_1A_1+\mathcal{X}\times_2A_2+\dots+\mathcal{X}\times_dA_d=\mathcal{C},
    \end{equation}
where $\times_i$ denotes the $i$th mode product for tensors (Definition~\ref{def:tensor_product}) and $A_i\in \C^{n_i\times n_i}$ are square matrices for each $i=1,\dots, d$. The unknown $\mathcal{X}$ and the right hand side  $\mathcal{C}$ are $d$ dimensional tensors of size $n_1\times \dots \times n_d$ and we assume that $\mathcal{C}$ is a low rank tensor in Tucker or TT format (see Section~\ref{sec:lr-tensors}).

The problem is equivalent to solving the linear system 
\begin{equation*}
    \vec{A}x=c 
\end{equation*}
    where $x$ and $c$ are vectorizations of $\mathcal{X}$ and $\mathcal{C}$ respectively, and
\begin{equation*}
    \vec{A}=\sum_{i=1}^d I_{n_d}\otimes \cdots \otimes I_{n_{i+1}} \otimes A_i \otimes I_{n_{i-1}}\otimes \cdots \otimes I_{n_1}, \quad \text{ with } \quad A_i\in \C^{n_i\times n_i}.
\end{equation*}
However, for large $d$, the solution of the linear system employing standard computational methods is unfeasible, because the size of the linear system grows exponentially in $d$.

One of the main applications of tensor Sylvester equations is the approximate solution of discretized PDEs, as shown in \cite{grasedyck2004existence}. Consider for instance the Poisson equation on a $d$-dimensional hypercube  
\begin{equation*}
    \begin{cases}
        -\Delta u = f &\text{ in }\Omega\\
        u\equiv 0 &\text{ on }\partial\Omega
    \end{cases}, \qquad \Omega= [0, 1]^d.
\end{equation*}
A discretization using finite differences produces a multilinear Sylvester equation in which the right hand side tensor is given by the sampling of the function $f$ on the discretization of the domain. If $f$ is a smooth multivariate function, then the right hand side can be well approximated by a tensor with low multilinear or Tensor Train rank, see \cite{shi2021compressibility}. In this setting the size of the right hand side is usually large, hence it is essential to exploit low rank structures.

In the case of $d=2$ the equation \eqref{eqn:tensor_sylv}, can be reformulated as the standard Sylvester equation
\begin{equation*}
    A_1X+XA_2^T=U_1U_2^H,
\end{equation*}
where $U_1\in \C^{n_1\times b}$ and $U_2\in \C^{n_2\times b}$, with $b\ll n_1,n_2$. This type of matrix equations has applications in control theory \cite{antoulas2001approximation, benner2015survey},  and it has been extensively studied in the literature, see for instance \cite{simoncini2016computational}. Moreover, in \cite{casulli2022} the authors employed block rational Krylov methods for solving Sylvester equations, that is, they solve the smaller size projected equation 
\begin{equation*}
    V_1^HA_1V_1Y+YV_2^HA_2^TV_2=V_1^HU_1U_2^HV_2,
\end{equation*}
where $V_1$ and $V_2$ are orthonormal block bases of the block rational Krylov subspaces $\rat_{k_1}(A_1,U_1,\vec{\xi}_1)$ and $\rat_{k_2}(A_2,U_2,\vec{\xi}_2)$, respectively (see Section~\ref{sec:block-rat-kry}), and then they approximate the solution $X$ by the matrix $V_1YV_2^H$. The authors also developed pole selection strategies and techniques based on pole reordering to efficiently compute the residual.
This work generalizes those ideas to the case of general $d$.
The case of tensor Sylvester equations has been studied by Kressner and Tobler in \cite{kressner2010krylov}, for the solution of the equation
\begin{equation*}
    \mathcal{X}\times_1A_1+\mathcal{X}\times_2A_2+\dots+\mathcal{X}\times_dA_d=c_1\times_2c_2\times_3\dots\times_d c_d, \quad \text{ with } c_i\in \C^{n_i},
\end{equation*}
projecting onto polynomial Krylov subspaces. This work extends this procedure to equations with a more general right hand side. The use of rational Krylov subspaces gives more freedom in the choice of the projection subspaces, through pole selection.

The rest of this paper is organized as follows: Section~\ref{sec:notation} contains preliminary definitions and results about matrix polynomials, block rational Krylov subspaces and tensors; Section~\ref{sec:tensor-krylov} is devoted to the introduction of tensorized block rational Krylov subspaces and their application for the solution of tensor Sylvester equations with right hand side with low multilinear or Tensor Train rank; Section~\ref{sec:pole_selection} discusses pole selection strategies and in Section~\ref{sec:residual} an efficient way to compute the residual is presented; finally, in Section~\ref{sec:num-exp} the developed methods are tested on the numerical solution of discretized PDEs.

\section{Notation and basic definitions}\label{sec:notation}

We use $\bar{\alpha}$ to denote the complex conjugate of $\alpha \in \C$ and $A^H$ to denote the conjugate transpose of a matrix $A\in\C^{n\times n}$. We denote by $\Lambda(A)$ the spectrum of $A$ and by $\W(A)$ its field of values, that is $\W(A)=\{x^HAx, \text{ with } x^Hx=1\}.$
    The set of extended complex numbers is denoted by $\overline{\C}$. We denote by $\pol(\C)$ and $\pol_k(\C)$ the space of polynomials and polynomials with degree bounded by $k$, respectively. For any polynomial $Q(z)$ we use $\bar{Q}(z)$ to denote the polynomial that has as coefficients the conjugates of the coefficients of $Q(z)$. 
    Given two vectors, $\vec k$, $\vec h$ with $d$ components, the notation $\vec h \le \vec k$ means that $\vec h$ is component-wise smaller than $\vec k$. The space of polynomials in $d$ variables with degree bounded by $\vec k=(k_1,\dots,k_d)$ is denoted by $\pol_{\vec k}(\C)$. The symbol $\pol_k(\C^{n\times n})$ is used to denote matrix polynomials of degree less than $k$, with coefficients in $\C^{n\times n}$, analogously, $\pol_{\vec k}(\C^{n\times n})$ denotes multivariate matrix polynomials of degree bounded by $\vec k$ and coefficients in $\C^{n\times n}$.  
    The identity matrix of size $s$ is denoted by $I_s$. We often use the terminology
    ``block vectors'', to indicate tall and skinny matrices. The size of blocks is denoted by $b$.  The Frobenius norm and the Euclidean norm are denoted by
    $\norm{\cdot}_F$ and $\norm{\cdot}_2$, respectively. We use the symbol
    $\otimes$ to denote the Kronecker product and the symbol $\text{vec}$ to denote the operator that vectorizes a tensor, that is, transforms a tensor into a vector obtained by ordering the elements of the tensor lexicographically.

\subsection{Matrix polynomials and rational functions}
In this section, we provide some
    definitions and properties about matrix polynomials that we use in the paper.  
    
    Let $\pol(\C^{b\times b})$ be the space of polynomials with coefficients in
    $\C^{b \times b}.$ We refer to these as matrix polynomials. We denote by
    $\pol_d(\C^{b\times b})$ the set of matrix polynomials of degree less or equal than $d$. A matrix polynomial is said to be monic if its leading
    coefficient is equal to the identity.

    Given a matrix polynomial $P(z)=\sum_{i=0}^d z^i \Gamma_i,$ where
    $\Gamma_i\in \C^{b\times b}$ for each $i$, we can define the operators $\circ$ and $\circ^{-1}$ from
    $\C^{n\times n}\times\C^{n\times b}$ to $\C^{n\times b}$ as follows: given
    two matrices $A\in \C^{n\times n}$ and $ v\in \C^{n\times b},$ we set
    \begin{equation*} 
    P(A)\circ  v :=\sum_{i=0}^d A^i v\Gamma_i \quad \text{ and } \quad             P(A)\circ^{-1} v=\text{vec}^{-1}\left(\left(\sum_{i=0}^d \Gamma_i^T\otimes A^{i}\right)^{-1}\text{vec} ( v)\right),
    \end{equation*} 
    where for the well posedness of $\circ^{-1}$ it is required $\det(P(\lambda))\neq 0$ for each $\lambda$ eigenvalue of $A$.
    
    \begin{remark}
        It holds
        \begin{equation*}
            P(A)\circ (P(A)\circ^{-1} v)= v, \quad \text{ and } \quad P(A)\circ^{-1} (P(A)\circ  v)= v.
        \end{equation*}
    \end{remark}
    
    These two operators can be extended to the case of rational matrices with 
    prescribed poles: let $Q(z)\in \pol(\C)$ and let $R(z)\in \pol(\C^{b\times b})/Q(z)$, that is there exists $P(z)\in \pol(\C^{b\times b})$ such that $R(z)=P(z)/Q(z)$; given $A\in\C^{n\times n}$ and $ v\in \C^{n\times b}$, we define
        \begin{equation*}
            R(A) \circ  v =Q(A)^{-1}\cdot P(A)\circ  v,
        \end{equation*}
        and
        \begin{equation*}
            R(A) \circ^{-1}  v =Q(A)\cdot P(A)\circ^{-1}  v.
        \end{equation*}    
    The representation of a rational matrix in the form 
    $R(z) = P(z) / Q(z)$ is not unique; however, two equivalent 
    representations yield the same linear mapping 
    $  v \mapsto R(A) \circ v$, and this makes the previous 
    definition well-posed.     
    For a more complete discussion, we refer to \cite{casulli2022}.

    Given a matrix polynomial $P(z)=\sum_{i=0}^dz^i\Gamma_i$, we denote by
    $P^H(z)$ the matrix polynomial $P^H(z): = \sum_{i=0}^dz^i\Gamma_i^H$. Analogously, given a function
    $R(z)=P(z)/Q(z)$, we denote by $R^H(z)$ the rational
    function $P^H(z)/\bar{Q}(z)$.
    
    Finally, given $A\in \C^{d b \times d b}$ and $ v\in \C^{d b\times b},$ a block characteristic polynomial of $A$ with respect to $ v$ is a matrix polynomial $P(z)\in \pol_d(\C^{b\times b})$ such that 
    \begin{equation*}
        P(A)\circ  v=0.
    \end{equation*}
We refer to \cite[Section~2.5]{lund2018new} for a more in-deep coverage of the topic.

\subsection{Block rational Krylov subspaces}\label{sec:block-rat-kry}

Given a matrix $A\in \C^{n \times n}$, a block vector $ v\in \C^{n\times b}$
and a sequence of poles $\vec{\xi}_{k}=\{\xi_j\}_{j=0}^{k-1}\subseteq \C\cup
\{\infty\}\setminus \Lambda(A)$ the $k$th block rational Krylov space is defined
as
\begin{equation*}  
\rat_k(A,v, \vec{\xi}_k) = \left\{ R(A) \circ v : R(z) = \frac{P(z)}{Q_{k}(z)}, \text{with } P(z)\in \pol_{k-1}(\C^{b\times b})\right\},
\end{equation*}
where $Q_{k}(z) = \prod_{\xi_j\in\vec{\xi}_k, \xi_j\neq \infty}(z - \xi_j)$. For
simplicity, we sometimes denote such space by $\rat_k(A,v)$ omitting poles.
Note that when choosing all poles equal to $\infty$ we recover the classical
definition of block Krylov subspaces.

It can be proved that if $\vec\xi_{k}\subsetneq \vec \xi_{k+1}$, then $\rat_k(A, v, \vec {\xi}_k)\subseteq \rat_{k+1}(A, v, \vec {\xi}_{k+1})$. In this
work, we will assume that the block rational Krylov subspaces are always
strictly nested, that is $\rat_k(A, v, \vec {\xi}_k)\subsetneq \rat_{k+1}(A, v,\vec {\xi}_{k+1})$ and
that the dimension of $\rat_k(A,v)$ is equal to $k b$. 

An orthonormal block basis of $\rat_{k}(A,v)$ (for simplicity, we will often
just say ``orthonormal basis'') is defined as a matrix $V_k=[ v_1,\dots,
 v_k]\in \C^{n\times bk}$ with orthonormal columns, such that every block
vector $v\in\rat_k(A,b)$ can be written as $v=\sum_{i=1}^k v_i
\Gamma_i$, for $\Gamma_i\in \C^{b\times b}.$ It can be computed by the block rational 
Arnoldi Algorithm\footnote{ For
simplicity we described a version of the algorithm that does not allow poles
equal to zero. For a more complete version of the algorithm, we refer to
\cite{elsworth2020block}.}
\refeq{algorithm:block-Arnoldi}, that iteratively computes the block columns of
$V_{k+1}$ and two matrices $\underline{K}_k, \underline{H}_k\in\C^{b(k+1) \times
bk}$ in block upper Hessenberg form such that 
\begin{equation}\label{eqn:rad0}
AV_{k+1}\underline{K_k}=V_{k+1}\underline{H_k}.
\end{equation}
We use the symbols $K_k$ and $H_k$ to denote the $bk\times bk$ head principal submatrices of $\underline{K_k}$ and $\underline{H_k}$, respectively. Moreover, we call ``Arnoldi iteration'' the part of Algorithm~\ref{algorithm:block-Arnoldi} enclosed between rows $4$ and $11$.

\begin{algorithm}
    \begin{algorithmic}[1]
	\Require{$A \in \C^{n \times n}, v \in \C^{n\times b},\vec {\xi}_{k+1}= \{\xi_0, \dots, \xi_{k}\}$}
	\Ensure{ $V_{k+1}\in \C^{n\times b(k+1)},$ $\underline{H}_k,  \underline{K}_k\in \C^{b(k+1)\times bk}$}

    \State $ w \gets (I-A/\xi_0)^{-1}v$ \Comment{with the convention $A/\infty=0$}
    \State $[ v_1, \sim ]\gets \text{qr}( w)$  \Comment{compute the thin QR decomposition}  

	\For{$j = 1, \dots, k$}
	\State  Compute $ w=(I-A/\xi_{j})A v_{j}$
    \For{$i = 1, \dots, j$}
    \State  $(\underline{H}_k)_{ix(i),ix(j)}\gets  {v}_i^H w$ \Comment{where $ix(s)=(s-1)b+1:sb$}
    \State$ w \gets  w- v_j(\underline{H}_k)_{ix(i),ix(j)}$
    \EndFor
    \State$[ v_{j+1}, (\underline{H}_k)_{ix(j+1),ix(j)} ]\gets \text{qr}( w)$  \Comment{compute the thin QR decomposition}  
    \State$(\underline{K}_k)_{ix(i),1:(j+1)b}\gets (\underline{H}_k)_{ix(i),1:(j+1)b}/\xi_{j} -{e}_j,$ \Comment{where ${e}_j=[0,\dots,0, I_b,0]^T$ }  
    \State $V_{j+1}\gets[ v_1,\dots, v_{j+1}]$
	\EndFor
\end{algorithmic}

	\caption{Block Rational Arnoldi} \label{algorithm:block-Arnoldi}

\end{algorithm}

\subsection{Low rank tensors}\label{sec:lr-tensors}
In this section we briefly recall basic concepts about tensors, focusing on the representation of low rank tensors in Tucker and Tensor Trains formats. A broader treatment of the argument can be found in \cite{kolda2009tensor} and \cite{oseledets2011tensor}.

The simplest way to define the rank of a $d$-dimensional tensor $\mathcal{X}\in \C^{n_1\times n_2\times \dots \times n_d}$, is the minimum number $k$ of ``rank one'' tensors which sum equals to $\mathcal{X}$ that is, denoting by $x$ a vectorization of $\mathcal{X}$, 
\begin{equation*}
    x=\sum_{i=1}^ku_{i,1}\otimes u_{i,2}\otimes\dots \otimes u_{i,d}, \quad \text{ with } u_{i,j}\in \C^{n_i} \text{ for each } i, j,
\end{equation*}
where $\otimes$ denotes the Kronecker product. This is called CP rank and the above representation of a tensor is called Canonical Polyadic decomposition (usually denoted by CP). The main issue of this decomposition is that the problem of determining the CP rank of a given tensor is NP-hard. To overcome this issue, several alternative definitions of rank have been introduced. In this work we focus on the concepts of multilinear and Tensor Trains ranks, starting by introducing a couple of related definitions.

\begin{definition}
     For each $i=1,\dots, d$ and $j=1,\dots, n_i$, the $j$th mode-$i$ fiber of a tensor $\mathcal{X}\in \C^{n_1\times n_2\times \dots \times n_d}$ is the $n_1\cdots n_{i-1}\cdot n_{i+1}\dots n_d$, vector that contains all the entries of $\mathcal{X}$ with $i$th index equal to $j$ ordered lexicographically. The mode-$i$ unfolding of $\mathcal{X}$, denoted by $X_{(i)}$, is the  $n_i\times n_1\cdots n_{i-1}\cdot n_{i+1}\cdots n_d$ matrix that has as $j$th row the transpose of the $j$th mode-$i$ fiber.
\end{definition}

The unfoldings can be used to define the multilinear rank.

\begin{definition}
    The multilinear rank of a tensor $\mathcal{X}$ is defined as the vector $\vec{k}=(k_1,\dots,k_d),$ where, for each $i=1,\dots, d$,  $k_i$ is the rank of the $i$th unfolding $\mathcal X_{(i)}$.
\end{definition}

As the CP rank is related to the CP decomposition, also the multilinear rank can be associated with a tensor decomposition, called Tucker decomposition. Before introducing this concept, we need to define how to multiply a tensor by a matrix.

\begin{definition}\label{def:tensor_product}
    The $i$th mode product of a tensor $\mathcal{X}\in \C^{n_1\times n_2\times \dots \times n_d}$ by a matrix $B\in \C^{n_i\times j},$ is the $n_1\dots\times n_{i-1}\times  j\times n_{i+1}\times \dots \times n_d$ tensor, denoted by $(\mathcal{X}\times_i B)$, defined as
    \begin{equation*}
        (\mathcal X \times_i B)_{s_1\dots s_d}=\sum_{t=1}^{n_i}\mathcal{X}_{s_1,\dots,s_{i-1},t,s_{i+1},\dots,s_d}B_{s_{i},t},
    \end{equation*}
    for each $s_i\in\{1,\dots,n_i\}$, $i\in\{1,\dots, d\}.$
\end{definition}

First introduced by Tucker in \cite{tucker1966some}, the Tucker decomposition decomposes a tensor $\mathcal{X}\in \C^{n_1\times n_2\times \dots \times n_d}$ into a core tensor $\mathcal G\in \C^{k_1\times k_2\times \dots \times k_d}$ multiplied by matrices $B_i\in \C^{n_i\times k_i}$ with orhonormal columns, along each mode $i$, that is,
\begin{equation}\label{eqn:Tucker_decomp}
    \mathcal{X}=\mathcal G\times_1 B_1\times_2 B_2\times_3\dots \times_dB_d.
\end{equation}
The generators of a Tucker decomposition are usually denoted by $\tuck{\mathcal G;B_1,\dots,B_d}.$ 

\begin{remark}
    The $i$th mode unfolding of the tensor \eqref{eqn:Tucker_decomp} can be written as 
    \begin{equation}\label{eqn:unfolding-tuck}
        \mathcal X_{(i)}=B_i\mathcal G_{(i)}(B_d\otimes \dots \otimes B_{i+1}\otimes B_{i-1}\otimes \dots \otimes B_1).
    \end{equation}
    In particular, the multilinear rank of $\mathcal X$ is component-wise smaller than $(k_1,\dots,k_d)$. See \cite[Section~4]{kolda2009tensor} for further details.
\end{remark}

Note that if $k_i \ll n_i$, for each $i$, the Tucker decomposition allows us to compress the data. For a given tensor, the quasi-optimal approximant in Tucker format with multilinear rank $(r_1,\dots,r_k$) can be computed by repeatedly truncating the $i$th mode unfoldings. This procedure is usually known as multilinear SVD, or high-order SVD (HOSVD), see \cite{de2000multilinear}.

We remark that the memory needed to store a tensor in Tucker format is $\mathcal{O}(k_1\cdots k_d+k_1n_1+\dots+k_dn_d)$, which is a great benefit with respect to storing the full tensor. However, the needed storage is exponential in the dimension of the tensor, hence this representation becomes unfeasible if $d$ is too large. To overcome this problem other low rank representations have been introduced, such as Tensor Trains introduced by Oseledets in \cite{oseledets2011tensor}.

Given a tensor $\mathcal{X}\in \C^{n_1\times\dots\times n_d},$ a Tensor Train decomposition (also called TT decomposition) consists in a sequence of tensors $\mathcal G_1\in \C^{n_1\times r_1}, \mathcal G_2 \in \C^{r_1\times n_2 \times r_2},\dots, \mathcal G_{d-1}\in \C^{r_{d-2}\times n_{d-1}\times r_{d-1}}, \mathcal G_d \in \C^{r_{d-1}\times n_d}$, called carriages, such that
\begin{equation*}
    \mathcal{X}_{(i_1,\dots,i_d)}=\sum_{s_1,\dots, s_{d-1}}{\mathcal{G}_1}_{(i_1,s_1)}{\mathcal{G}_{2}}_{(s_1,i_2,s_2)}\cdots{\mathcal{G}_{d-1}}_{(s_{d-2},i_{d-1},s_{d-1})} {\mathcal{G}_d}_(s_{d-1},i_d).
\end{equation*}
for each $(i_1,\dots,i_d)\le (n_1,\dots,n_d)$. The numbers $r_1,\dots,r_{d-1}$ are called ranks of the decomposition. 

For each $i=1,\dots, d-1$, let $X^{\{i\}}\in \C^{n_1\cdots n_i\times n_{i+1}\cdots n_d}$  be the matrix obtained by grouping the first $i$ indices of a tensor $\mathcal X$ as row indices, and the remaining ones as column indices. The TT rank of a tensor is defined as follows.

\begin{definition}
    Given a tensor $\mathcal{X}\in \C^{n_1\times\dots\times n_d},$ the vector $(r_1,\dots,r_{d-1})$, where $r_i$ is the rank of $X^{\{i\}}$, is called Tensor Train rank (sometimes denoted by TT rank) of $\mathcal X$. 
\end{definition}

The definition of TT rank and TT decomposition are closely related, in particular for each tensor there exists a Tensor Train decomposition with ranks component-wise smaller or equal than its Tensor Train rank (see \cite[Theorem~2.1]{oseledets2011tensor}). 

In practice, for any $\epsilon>0$, every tensor $\mathcal Y$ can be approximated by a tensor $\mathcal X$ in TT format with relative accuracy $\epsilon$, i.e., 
\begin{equation*}
    \norm{\mathcal X-\mathcal Y}_F\le \epsilon \norm{\mathcal Y}_F
\end{equation*} 
employing the TT-SVD algorithm. We refer to \cite{oseledets2011tensor} for further details.

We conclude this section deriving a low rank representation for $\mathcal{X}_{(i)},$ where $\mathcal X$ is a tensor in TT format with carriages $\ttrain{\mathcal G_1,\dots,\mathcal{G}_d}.$ 

First of all, we notice that any entry of $\mathcal{X}$ can be written as 
\begin{equation*}
    \mathcal{X}_{(i_1,\dots,i_d)}=A_{(i_1,\dots,i_{j-1})}G_j(i_j)B_{(i_{j+1},\dots,i_d)},
\end{equation*}
or equivalently
\begin{equation}\label{eqn:kron-entry-tt}
    \mathcal{X}_{(i_1,\dots,i_d)}=\left(B_{(i_{j+1},\dots,i_d)}^T\otimes A_{(i_1,\dots,i_{j-1})}\right)\text{vec}(G_j(i_j)),
\end{equation}
 where $A_{(i_1,\dots,i_{j-1})}$ and $B_{(i_{j+1},\dots,i_d)},$ are a row and a column vector, respectively, defined as  
\begin{equation*}
    (A_{(i_1,\dots,i_{j-1})})_h= \begin{cases} \sum_{s_1,\dots, s_{j-2}}{\mathcal{G}_1}_{(i_1,s_1)}{\mathcal{G}_{2}}_{(s_1,i_2,s_2)}\cdots{\mathcal{G}_{j-1}}_{(s_{j-2},i_{j-1},h)} &\text {for } j>1,\\
        1& \text {otherwise,}
    \end{cases},
\end{equation*}
\begin{equation*}
    (B_{(i_{j+1},\dots,i_d)})_k=\begin{cases}\sum_{s_{j+1},\dots, s_{d-1}}{\mathcal{G}_{j+1}}_{(k,i_{j+1},s_{j+1})}\cdots{\mathcal{G}_{d-1}}_{(s_{d-2},i_{d-1},s_{d-1})} {\mathcal{G}_d}_(s_{d-1},i_d) &\text {for } j<d,\\
        1& \text {otherwise}
    \end{cases}
\end{equation*}
and $G_j(i_j)$ is a matrix defined as
\begin{equation*}
    G_j(i_j)_{h,k}=\begin{cases}
        \mathcal{G}_1(i_1,k) &\text{if }j=1,\\
        \mathcal{G}_d(h,i_d) &\text{if }j=d,\\
        \mathcal{G}_j(h,i_j,k) &\text{otherwise.}
    \end{cases}
\end{equation*}
Noting that 
\begin{equation*}
    \begin{bmatrix}
        \text{vec}(G_j(1))^T\\
        \text{vec}(G_j(2))^T\\
        \vdots\\
        \text{vec}(G_j(n_j))^T\end{bmatrix}=
        \begin{cases}\mathcal{G}_1 &\text{if }j=1,\\
            (\mathcal{G}_j)_{(2)} &\text{otherwise,}
        \end{cases} 
\end{equation*}
from \eqref{eqn:kron-entry-tt} we have 
\begin{equation}\label{eqn:unfolding-tt}
    \mathcal{X}_{(j)}=\begin{cases}
        \mathcal{G}_1C &\text{if }j=1,\\
        (\mathcal{G}_j)_{(2)}C &\text{otherwise,}
    \end{cases}
\end{equation}
where $C$ is the block row that has as cloumns the vectors $B_{(i_{j+1},\dots,i_d)}\otimes A^T_{(i_1,\dots,i_{j-1})}$, ordered lexicographically with respect to $(i_1,\cdots,i_d).$







\section{Tensorized Krylov methods} \label{sec:tensor-krylov}
Employed by Kressner and Tobler in \cite{kressner2010krylov} for solving tensor Sylvester equations 
\begin{equation}\label{eqn:CP-Sylv}
    \mathcal{X}\times_1A_1+\mathcal{X}\times_2A_2+\dots+\mathcal{X}\times_dA_d=c_1\times_2c_2\times_3\dots\times_d c_d, \quad \text{ with } c_i\in \C^{n_i},
\end{equation}
the tensorized Krylov subspaces are defined as
\begin{equation*}
    \mathcal{K}_{\vec k}^{\otimes}(\{A_i\}_i,\{c_i\}_i)=\text{span}(\rat_{k_1}(A_1,c_1,\vec{\infty})\otimes \rat_{k_1}(A_d,c_d,\vec{\infty})\otimes\dots\otimes\rat_{k_d}(A_d,c_d,\vec{\infty})),
\end{equation*}
where $\vec k=(k_1,\dots,k_d)$ and $\vec{\infty}=\{\infty,\dots,\infty\}.$

These subspaces can be described also using multivariate polynomials, as it is stated in the next lemma (\cite[Lemma~3.2]{kressner2010krylov}).

\begin{lemma}
    Let $\pol_{\vec k}(\C)$ be the space of multivariate polynomials with degree bounded by $\vec k$. We have 
    \begin{equation*}
        \mathcal{K}_{\vec k}^{\otimes}(\{A_i\}_i,\{c_i\}_i)=\{p(A_1,...
        ,A_d)(c_1\otimes\dots
        \otimes c_d),\text{ for }p\in\pol_{\vec k}(\C)\},
    \end{equation*}
    where for each $p=\sum_{I=(i_1,\dots,i_d)\le\vec k}c_Ix^{i_1}\cdots x^{i_d}\in \pol_{\vec k}(\C)$,
    \begin{equation*}
        p(A_1,\dots,A_d)=\sum_{I=(i_1,\dots,i_d)\le\vec k}c_IA_1^{i_1}\otimes\dots \otimes A_d^{i_d}.
    \end{equation*}
\end{lemma}

The algorithm for solving \eqref{eqn:CP-Sylv} consists in solving the projection of the equation into the tensorized Krylov subspace, that is
\begin{equation*}
    \mathcal{Y}_{\vec k}\times_1V_1^HA_1V_1+\mathcal{Y}_{\vec k}\times_2V_2^HA_2V_2+\dots+\mathcal{Y}_{\vec k}\times_dV_d^HA_dV_d=V_1^Hc_1\times_2V_2^Hc_2\times_3\dots\times_d V_d^Hc_d,
\end{equation*}
where, for each $i$, $V_i$ is an orthonormal basis of the polynomial Krylov subspace $\rat_{k_i}(A_i,c_i)$, and in approximating the solution $\mathcal{X}$ by the low multilinear rank tensor 
\begin{equation*}
    {\mathcal{X}_{\vec k}}=\mathcal{Y}_{\vec k}\times_1V_1\times_2V_2\dots\times_3\times_dV_d.
\end{equation*}

 The authors have also proved that the solution $\mathcal{X}$ can be well approximated by a low rank tensor, relating the norm of the error with the approximation of the function $\frac{1}{x_1+x_2\dots+x_d}$ with a sum of separable multivariate functions, see \cite[Theorem~2.5]{kressner2010krylov}. Moreover, they also analyzed the effects of using extended Krylov subspaces (i.e., $\rat_{k_i}(A_i,c_i,\vec{\xi})$, where $\vec{\xi}$ is given by alternating $0$ and $\infty$), in the construction of tensorized Krylov subspaces.

 In the next sections, we generalize such procedure to the solution of the tensor Sylvester equation
\begin{equation}\label{eqn:tensor_sylv3}
    \mathcal{X}\times_1A_1+\mathcal{X}\times_2A_2+\dots+\mathcal{X}\times_dA_d=\mathcal{C},
\end{equation}
for $\mathcal{C}$ with low multilinear or TT rank employing  as $V_i$, an orthonormal basis for the block rational Krylov subspaces $\rat_{k_i}(A_i,C_i)$ with appropriate block vectors $C_i$, for $i=1,\dots, d$.

\subsection{Tensorized block rational Krylov methods}
One of the novelties of this work is to analyze the use of block rational Krylov subspaces in tensorized Krylov methods. On one hand, the use of block Krylov subspaces for solving the tensor Sylvester equation \eqref{eqn:tensor_sylv3} allows us to easily treat the case of $C_i$ with more than one column. On the other hand, the use of rational Krylov methods gives more freedom in the choice of the projection subspace, through the pole selection.

 We start defining \tbrk subspaces.

\begin{definition}{\label{def:tensor_rational_krylov}}
    For each $i=1,\dots d$, let $A_i\in \C^{n_i\times n_i}$ and $C_i\in \C^{n_i\times b_i}.$ Let $\vec k=(k_1,\dots,k_d)$, with $k_i\in \mathbb{N}$ and for each $i=1,\dots, d$ let $\vec{\xi}_i\in \bar{\C}^{k_i}$. We define the tensorized block rational Krylov subspace associated with $ \{A_i\}_i , \{C_i\}_i $ and $\{\vec \xi_i\}_i$ as
    \begin{equation*}
        \rat_{\vec k}^{\otimes}( \{A_i\}_i, \{C_i\}_i , \{\vec\xi_i\}_i)=\left\{\sum_{i=1}^s v_i \Gamma_i, \text{ for } s\in \mathbb{N}, v_i\in W, \Gamma_i \in \C^{b_1\cdots b_d\times b_1\cdots b_d}\text{ for each }i\right\},
    \end{equation*}
    with $W=\rat_{ k_1}( A_1, C_1 , \xi_1)\otimes \rat_{ k_2}( A_2, C_2 , \xi_2)\otimes \dots \otimes \rat_{ k_d}( A_d, C_d, \xi_d)$.
\end{definition}
For simplicity of notation, we sometimes omit poles, denoting a tensorized block rational Krylov subspace just by $\rat_{\vec k}^{\otimes}( \{A_i\}_i, \{C_i\}_i)$.

The relation between rational Krylov spaces and rational functions can be extended also in the case of \tbrk spaces. First of all, we define an extension of the operator $\circ$ to multivariate polynomials.

\begin{definition}
    Let
    \begin{equation*}
        P(x_1,\dots,x_d)= \sum_{I=(i_1,\dots,i_d)\le \vec k}\Gamma_I x_1^{i_1}\cdot x_2^{i_2}\cdots x_d^{i_d}\in\pol_{\vec{k}}(\C^{b_1\cdots b_d\times b_1\cdots b_d}),
    \end{equation*}
     and let $A_i\in\C^{n_i \times n_i}$ , $C_i\in\C^{n_i\times b_i}$ for $i=1,\dots,d$. We define
     \begin{equation*}
        P(A_1,\dots,A_d)\circ(C_1,\dots,C_d)=\sum_{I=(i_1,\dots,i_d)\le \vec k} (A_1^{i_1}C_1\otimes A_2^{i_2}C_2\otimes \cdots \otimes A_d^{i_d}C_d)\Gamma_I.
     \end{equation*}
     Moreover, if $R(x_1,\dots,x_d)=P(x_1,\dots,x_d)/Q_1(x_1)\cdots Q_d(x_d)$ with $Q_i(x)\in \pol(\C)$ for each $i$,  we define
     \begin{equation*}
        R(A_1,\dots,A_d)\circ(C_1,\dots,C_d)=Q_1(A_1)^{-1}\otimes\cdots \otimes Q_d^{-1}(A_d)\cdot (P(A_1,\dots,A_d)\circ(C_1,\dots,C_d)).
     \end{equation*}

\end{definition}

Now we can state the following lemma.

\begin{lemma}
    Let $\pol_{\vec k}(\C^{b_1\cdots b_d\times b_1\cdots b_d})$ be the space of matrix polynomials in $d$ variables of degree bounded by $\vec k$. It holds
    \begin{equation*}
        \rat_{\vec k}^{\otimes}( \{A_i\}_i, \{C_i\}_i , \{\vec\xi_i\}_i)=
    \{r(A_1,\dots, A_d)\circ (C_1,\dots,C_d) :r\in \pol_{\vec k}/Q(x_1)\cdots Q(x_d)\},
    \end{equation*}
    with $Q_i(x)=\prod_{\xi\in \vec {\xi}_i, \xi\neq \infty}(x-\xi)$.
    \end{lemma}
    \begin{proof} Let $v\in \{\rat_{ k_1}( A_1, C_1 , \vec {\xi}_1)\otimes \rat_{ k_2}( A_2, C_2 , \vec {\xi}_2)\otimes \dots \otimes \rat_{ k_d}( A_d, C_d, \vec {\xi}_d)\}.$ For each $i$ there exists a univariate rational function $r_i(x)=p_i(x)/Q_i(x)$ with $p_i(x)=\sum_{j=0}^{k_i}\Gamma_j^{(i)}x^j\in \pol_{k_i}(\C^{b_i\times b_i})$, such that
        \begin{equation*}
            v=r_1(A_1)\circ C_1 \otimes r_2(A_2)\circ C_2\otimes\dots\otimes r_d(A_d)\circ C_d.
        \end{equation*}
    Denoting by $r_v=\sum_{I\le \vec k}\Gamma_I x_1^{i_1}\cdots x_d^{i_d}/Q_1(x_1)\cdots Q_d(x_d)$, where $\Gamma_I= \Gamma_{i_1}^{(1)}\otimes \Gamma_{i_2}^{(2)}\otimes \dots \otimes \Gamma_{i_d}^{(d)}$, with $I=(i_1,\dots,i_d)$, it is immediate to verify that 
    \begin{equation*}
        v=r_v(A_1,A_2,\dots,A_d)\circ (C_1,C_2,\dots,C_d).
    \end{equation*}
    Hence, if $w=\sum_{i=1}^s v_i \Delta_i$, with $v_i\in \{\rat_{ k_1}( A_1, C_1 , \vec {\xi}_1)\otimes \rat_{ k_2}( A_2, C_2 , \vec {\xi}_2)\otimes \dots \otimes \rat_{ k_d}( A_d, C_d, \vec {\xi}_d)\}$ and $\Delta_i\in \C^{b_1\cdots b_d\times b_1\cdots b_d}$ for each $i$,
    then letting 
    \begin{equation*}r_w=\sum_{i=1}^s \Delta_i r_{v_i}\in \pol_{\vec k}(\C^{b_1\cdots b_d \times b_1 \cdots b_d})/Q(x_1)\cdots Q(x_d),
    \end{equation*}
     we have 
    \begin{equation*}
        w=r_w(A_1,A_2,\dots,A_d)\circ (C_1,C_2,\dots,C_d).
    \end{equation*}
        To prove the other inclusion let $r=p(x_1,\dots, x_d)/Q_1(x_1)\cdots Q_d(x_d)$, where $p(x_1,\dots, x_d)=\sum_{I\le \vec k}\Gamma_I x_1^{i_1}\cdots x_d^{i_d}\in \pol_{\vec k}(C^{b_1\cdots b_d \times b_1\cdots b_d})$. It is easy to prove that
    \begin{equation*}
        \text{span}\{\C^{b_1\times b_1}\otimes \C^{b_2\times b_2}\otimes \dots \otimes \C^{b_d\times b_d}\}=\C^{b_1\cdots b_d\times b_1\cdots b_d},
    \end{equation*}
    in particular each $\Gamma_I$ can be written as 
    \begin{equation*}
        \Gamma_I=\sum_{j=1}^{t_I} \alpha_{j,I} \Gamma_{j,I}^{(1)}\otimes \Gamma_{j,I}^{(2)}\otimes \dots \otimes \Gamma_{j,I}^{(d)},
    \end{equation*}
    for $\alpha_{j,I}\in \C$ and $\Gamma_{j,I}^{(s)}\in \C^{b_s\times b_s}$ for each $s$, hence $r(A_1,\dots,A_d)\circ(C_1,\dots,C_d)$ equals to 
    \begin{equation*}
        \sum_{I\le \vec k}\sum_{j=1}^{t_I} \alpha_{j,I} Q_1(A_1)^{-1}A_1^{i_1}C_1\Gamma_{j,I}^{(1)}\otimes Q_2(A_2)^{-1}A_2^{i_2}C_2\Gamma_{j,I}^{(2)}\otimes \dots \otimes Q_d(A_d)^{-1}A_d^{i_d}C_d\Gamma_{j,I}^{(d)},
    \end{equation*}
  that is a linear combination of elements in $\{\rat_{ k_1}( A_1, C_1 , \vec {\xi}_1)\otimes \rat_{ k_2}( A_2, C_2 , \vec {\xi}_2)\otimes \dots \otimes \rat_{ k_d}( A_d, C_d, \vec {\xi}_d)\}$.
        
    \end{proof}

Notice that an orthonormal block basis for a tensorized block rational Krylov subspace is given by  $\vec V=\bigotimes_{i=1}^dV_i$, where $V_i$ is an orthonormal basis for $\rat_{ k_i}( A_i, C_i, \vec {\xi}_i)$, hence the computation of $\vec V$ reduces to the computation of $d$ block rational Krylov subspaces.

The space defined above can be used to solve a tensor Sylvester equation
\begin{equation}\label{eqn:tensor_sylv2}
    \mathcal{X}\times_1A_1+\mathcal{X}\times_2A_2+\dots+\mathcal{X}\times_dA_d=\mathcal{C},
\end{equation}
where $\mathcal C\in \C^{n_1\times n_2\dots\times n_d}$ has low multilinear or Tensor Train rank, using projection methods.

First of all, we have to compute $V_1,\dots, V_d$, orthonormal basis of $\rat_{k_1}(A_1,C_1)$,\dots, $\rat_{k_d}(A_d,C_d)$, respectively, employing the block rational Arnoldi algorithm (i.e., Algorithm~\ref{algorithm:block-Arnoldi}). The choice of the block vectors $C_i$ depends on the low rank representation of $\mathcal{C}$. This aspect is discussed in Section~\ref{sec:Tucker} and Section~\ref{sec:TT} for $\mathcal{C}$ in Tucker and Tensor Trains format, respectively. 

As in the classical Krylov tensor method, the solution $\mathcal{X}$ is approximated by the tensor $\mathcal{X}_{\vec k}=\mathcal{Y}_{\vec k}\times_1V_1\times_2\dots\times_dV_d,$ where $\mathcal{Y}_{\vec k}$ solves the smaller size tensor Sylvester equation
\begin{equation}\label{eqn:small_sylv}
    \mathcal{Y}_{\vec k}\times_1A_1^{(k_1)}+\dots+\mathcal{Y}_{\vec k}\times_dA_d^{(k_d)}=\mathcal{C}_{\vec k},
\end{equation}
with $A_i^{(k_i)}=V_i^HA_iV_i$ for each $i=1,\dots, d$ and $\mathcal{C}_{\vec k}=\mathcal{C}\times_1V_1^H\times_2 \dots \times_dV_d^H$. 
This choice satisfies the Galerkin condition
\begin{equation*}
    \mathcal{C}_{\vec k}-\mathcal{Y}_{\vec k}\times_1A_1^{(k_1)}+\dots+\mathcal{Y}_{\vec k}\times_dA_d^{(k_d)}\perp \rat^{\otimes}_{\vec k}(\{A_i\}_i,\{C_i\}_i).
\end{equation*}

\begin{remark}
    The solvability of \eqref{eqn:tensor_sylv2} does not guarantee the solvability of the projected equations \eqref{eqn:small_sylv}. A sufficient condition to avoid this issue is to require
        \begin{equation*}
        0\notin \W(A_1)+\W(A_2)+\dots+\W(A_d).
    \end{equation*} 
    However, this condition can be hard to verify. In practice, if a projected equation is not solvable, we can just change the projection space, for instance, using different poles.
\end{remark}

\subsection{Convergence analysis}\label{sec:convergence-analysis}
In the following, we combine the results from \cite{beckermann2013error} and \cite{casulli2022} to analyze the convergence of tensorized block rational Krylov methods. The outcomes of this section are fundamental in developing efficient ways to adaptively determine poles for the method and to compute the residual, topics that are extensively discussed in Sections~\ref{sec:pole_selection} and \ref{sec:residual}.

To easily apply the results of \cite{beckermann2013error}, we consider the tensor Sylvester equation in vectorized form, that is
\begin{equation}\label{eqn:vectorized_Sylv}
    \vec{A}x=c 
\end{equation}
    where $x$ and $c$ are vectorizations of $\mathcal{X}$ and $\mathcal{C}$ respectively, and
\begin{equation}\label{eqn:kron_A}
    \vec{A}=\sum_{i=1}^d I_{n_d}\otimes \cdots \otimes I_{n_{i+1}} \otimes A_i \otimes I_{n_{i-1}}\otimes \cdots \otimes I_{n_1}, \quad \text{ with } \quad A_i\in \C^{n_i\times n_i}.
\end{equation}

We define $\vec{V}= {V}_d\otimes V_{d-1} \otimes \cdots \otimes V_1,$ and for each $i=1,\dots, d$,
    \begin{align*}
        \vec{V}_i&=I_{n_d}\otimes \cdots \otimes I_{n_{i+1}} \otimes V_i \otimes I_{n_{i-1}}\otimes \cdots \otimes I_{n_1},\\
       \overline{\vec{V}}_i&= {V}_d\otimes \cdots \otimes V_{i+1} \otimes I_{n_i}\otimes V_{i-1}\otimes \cdots \otimes V_1.
    \end{align*}
    We  denote by $r(\vec{V},A_1,\dots,A_d,c)$, sometimes abbreviated by $r$, the residual $c -\vec{A}\vec{V}y$, where $y$ is the vectorization of the tensor $\mathcal{Y}_{\vec k}$ that solves the projected equation \eqref{eqn:small_sylv}. Analogously,
    \begin{equation*}
        r(\vec{V}_i, A_1^{(k_{1})},\dots A^{(k_{i-1})}_{i-1}, A_{i}, A_{i+1}^{(k_{i+1})},\dots,  A^{(k_{d})}_d,\overline{\vec{V}}_i^H c),
    \end{equation*} is defined as 
    \begin{equation*}
        \overline{\vec{V}}_i^H c-\left( A_d^{(k_{d})}\otimes \cdots \otimes A_{i+1}^{(k_{i+1})}\otimes A_{i}\otimes A_{i-1}^{(k_{i-1})}\otimes\cdots \otimes  A_1^{(k_{1})}\right)\vec{V}_iy.
    \end{equation*}

    To describe a representation of the residual that depends on the poles of the rational Krylov subspaces, we start by considering Proposition~2.2 of \cite{beckermann2013error}.

    \begin{proposition} \label{prop: beck-kress-tob} With the notation introduced above, the following statements hold:
        \begin{enumerate}
            \item The residual $r=c- \vec{A}y$ can be represented as
            \begin{equation*}
             r=\sum_{i=1}^d\overline{\vec{V}}_i r(\vec{V}_i, A_1^{(k_{1})},\dots A_{i-1}^{(k_{i-1})}, A_{i}, A_{i+1}^{(k_{i+1})},\dots,  A_d^{(k_{d})},\overline{\vec{V}}_i^Hc)+\hat{c},
        \end{equation*}
        where the remainder term $\hat{c}=(\prod_{i=1}^d(I-\overline{\vec{V}}_i\overline{\vec{V}}_i^H))c$ vanishes for $c \in \text{span}(\vec{V})$;
        \item The vectors $\hat{c}$ and 
        $
            \overline{\vec{V}}_i\overline{\vec{V}}_i^Hr=\overline{\vec{V}}_i r(\vec{V}_i, A_1^{(k_{1})},\dots A_{i-1}^{(k_{i-1})}, A_{i}, A_{i+1}^{(k_{i+1})},\dots, A_d^{(k_{d})},\overline{\vec{V}}_i^Hc),       
        $ for $i=1,\dots, d$ are mutually orthogonal. In particular, this implies 
        \begin{equation*}
            \norm{r}_2^2=\sum_{i=1}^d \norm{r(\vec{V}_i, A_1^{(k_{1})},\dots,  A_{i-1}^{(k_{i-1})}, A_{i}, A_{i+1}^{(k_{i+1})},\dots, A_d^{(k_{d})},\overline{\vec{V}}_i^Hc)}_2^2+\norm{\hat{c}}_2^2.
        \end{equation*}
        \end{enumerate}

    \end{proposition}

    Thanks to the previous proposition,  to monitor the norm of the residual it is sufficient to control the norms of $\hat {c}$ and  $\overline{\vec{V}}_i^H r$.

    For each $i$, the partial residual $\overline{\vec{V}}_i^H r$ is the vectorization of the tensor
    \begin{equation*}
        \mathcal{R}_i=\overline{\mathcal{C}}_i-\overline{\mathcal{Y}}^i_{\vec k}\times_1A_1^{(k_1)}-\dots-\overline{\mathcal{Y}}^i_{\vec k}\times_{i+1}A_{i-1}^{(k_{i-1})}-\overline{\mathcal{Y}}^i_{\vec k}\times_{i}A_{i}-\overline{\mathcal{Y}}^i_{\vec k}\times_{i+1}A_{i+1}^{(k_{i+1})}-\overline{\mathcal{Y}}^i_{\vec k}\times_dA_d^{(k_d)},
    \end{equation*}
    where $\overline{\mathcal{C}}_i= \mathcal{C}\times_1V_1^H\times_2\dots \times_{i-1}V_{i-1}^H\times_{i+1}V_{i+1}^H\times_{i+2}\dots \times_d V_d^H$ and $\overline{\mathcal{Y}}_{\vec k}^i={\mathcal{Y}}_{\vec k}\times_iV_i$. In particular, the Euclidean norm of $\overline{\vec{V}}_i^H r$ equals to the Frobenius norm of the $i$th mode unfolding
    \begin{equation}\label{eqn:residual_unfolding}
        (\mathcal{R}_i)_{(i)}=(\overline{\mathcal{C}}_i)_{(i)}-A_i(\overline{\mathcal{Y}}_{\vec k}^i)_{(i)}-(\overline{\mathcal{Y}}_{\vec k}^i)_{(i)}B_i,
    \end{equation}
    where
    \begin{equation}\label{eqn:B_i}
     \begin{split}
         B_i&=\sum_{j=1}^{i-1}I_{k_d}\otimes \cdots \otimes I_{k_{i+1}}\otimes I_{k_{i-1}}\otimes \cdots \otimes I_{k_{j+1}} \otimes {A_j}^{(k_{j})}\otimes I_{k_{j-1}}\otimes\cdots \otimes I_{k_{i_1}} \\
         &+\sum_{j=i+1}^{d}I_{k_d}\otimes \cdots \otimes I_{k_{j+1}} \otimes {A_j}^{(k_{j})}\otimes I_{k_{j-1}}\otimes\cdots \otimes I_{k_{i+1}}\otimes I_{k_{i-1}}\otimes \cdots \otimes I_{k_{i_1}}.
     \end{split}
\end{equation}
\begin{remark}\label{rmk:sylv_residual}
   The matrix  ${(\mathcal{R}_i)_{(i)}}$
 is the residual of the Sylvester equation $A_iX-XB_i=(\overline{\mathcal{C}}_i)_{(i)}$, solved projecting $A_i$ into the block rational Krylov subspace $\rat_{k_i}(A,C_i).$ 
\end{remark}

Summarizing, we have the following corollary of Proposition~\ref{prop: beck-kress-tob}. 

 \begin{corollary}\label{cor:residual-norm} The squared Euclidean norm of the residual $r(\vec{V},A_1,\dots,A_d,c)$ can be written as
    \begin{equation*}
        \norm{r(\vec{V},A_1,\dots,A_d,c)}_2^2
        =\sum_{i=1}^d \norm{(\mathcal{R}_i)_{(i)}}_F^2+\norm{\hat{c}}_2^2,
    \end{equation*}
    where the remainder term $\hat{c}$ vanishes for $c \in \text{span}(\vec{V})$.

 \end{corollary}

 \subsection{RHS in Tucker format}\label{sec:Tucker}
In this section we assume that the right hand side $\mathcal{C}$ of \eqref{eqn:tensor_sylv2} is given in Tucker format, generated by $\tuck{\mathcal{G}, U_1 ,\dots, U_d}$, with $\mathcal G\in \C^{b_i\times\dots \times b_d}$ and $U_i\in \C^{n_i\times b_i}$ for each $i=1,\dots, d$. 

For each $i$, a Tucker representation of the tensor ${\overline{\mathcal{C}}_i}$ is generated by 
\begin{equation*}
    \tuck{\mathcal{G}, V_1^HU_1,\dots,V_{i-1}^HU_{i-1},U_i,V_{i+1}^HU_{i+1}\dots,V_d^HU_d},
\end{equation*} hence, from \eqref{eqn:unfolding-tuck}, we have that the matrix $({\overline{\mathcal{C}}_i})_{(i)}$ admits the low rank representation 
$
    ({\overline{\mathcal{C}}_i})_{(i)}=U_iZ^H,
$
for an appropriate block vector $Z$. From Corollary~\ref{cor:residual-norm}, the convergence of the method is related to the norms of the matrices ${(\mathcal{R}_i)_{(i)}}$. By Remark~\ref{rmk:sylv_residual} we are implicitly solving the Sylvester equation $A_iX-XB_i=U_iZ^H$, by projecting $A_i$ into the block rational Krylov subspace $\rat_{k_i}(A,C_i).$ Hence, the natural choice of the block vector for the construction of the $i$th block rational Krylov subspace is $C_i=U_i$.

Assume now we know $V_1,\dots,V_d$ orthonormal basis for $\rat_{k_1}(A_1,U_1),\dots,\rat_{k_d}(A_d,U_d)$, respectively, and the projected matrices $ A_i^{(k_i)} =V_i^HA_iV_i$. We have to solve the projected tensor Sylvester equation
\begin{equation*}
    \mathcal{Y}_{\vec k}\times_1 A^{(k_1)}_1+\mathcal{Y}_{\vec k}\times_2 A^{(k_2)}_2+\dots+\mathcal{Y}_{\vec k}\times_d A^{(k_d)}_d={\mathcal{C}_{\vec k}},
\quad \text{ where } \quad
   { \mathcal C}_{\vec k}= \mathcal C \times_1 V_1^H,\dots,\times_dV_d^H.
\end{equation*}

Note that $\mathcal{C}_{\vec k}\in \C^{b_1k_1\times\dots\times b_dk_d}$, hence it is reasonable that such tensor can be fully stored and the solution $\mathcal Y_{\vec k}$ of the projected equation can be computed by a direct method such as the one presented by Chan and Kressner in \cite{chen2020recursive}. A Tucker decomposition of the approximate solution $\mathcal X_{\vec k}$ related with the \tbrk subspace is generated by $\tuck{\mathcal{Y}_{\vec k}, V_1,\dots,V_d}$.

\subsection{RHS in Tensor Train format}\label{sec:TT}
The main advantage of having $\mathcal{C}\in \C^{n_1\times\dots\times n_d}$ in TT format is the possibility of handling more summands in the tensor Sylvester equation since the memory storage in this format increases only linearly with $d$. Clearly, in such a case it is necessary to produce an approximate solution tensor $\mathcal{X}_{\vec k}$ in TT format as well.

Assume now that the tensor $\mathcal{C}$ is represented in TT format with carriages $\ttrain{\mathcal G_1,\dots,\mathcal G_d}$. For each $i$, a TT representation of the tensor ${\overline{\mathcal{C}}_i}$ is given by the carriages
\begin{equation*}
    \ttrain{\mathcal G_1\times_1V_1^H,\mathcal G_2\times_2V_2^H,\dots, \mathcal G_{i-1}\times_2V_{i-1}^H,\mathcal G_i,\mathcal G_{i+1}\times_2V_{i+1}^H,\dots,\mathcal G_{d}\times_2V_{d}^H}
\end{equation*} 
and from \eqref{eqn:unfolding-tt} we have that the matrix $({\overline{\mathcal{C}}_i})_{(i)}$ admits the low rank representation 
\begin{equation*}    ({\overline{\mathcal{C}}_i})_{(i)}=
    \begin{cases}
        \mathcal{G}_1Z^H &\text{if }i=1,\\
        (\mathcal{G}_i)_{(2)}Z^H &\text{otherwise,}
    \end{cases}
\end{equation*}
for an appropriate block vector $Z$. With the same argument of the Tucker case, we have that the natural choice of the block vector for the construction of the $i$th block
rational Krylov subspace is 
\begin{equation*}   C_i=
    \begin{cases}
        \mathcal{G}_1 &\text{if }i=1,\\
        (\mathcal{G}_i)_{(2)} &\text{otherwise.}
    \end{cases}
\end{equation*}

Assuming to know $V_1,\dots,V_d$ orthonormal bases for $\rat_{k_1}(A_1,C_1),\dots,\rat_{k_d}(A_d,C_d)$, respectively, and the projected matrices $ A^{(k_i)}_i =V_i^HA_iV_i$, we have to solve the projected tensor Sylvester equation
\begin{equation*}
    \mathcal{Y}_{\vec k}\times_1 A^{(k_1)}_1+\mathcal{Y}_{\vec k}\times_2A^{(k_2)}_2+\dots+\mathcal{Y}_{\vec k}\times_d A^{(k_d)}_d={\mathcal{C}_{\vec k}},
\quad \text{ where } \quad
   { \mathcal C}_{\vec k}= \mathcal C \times_1 V_1^H\dots\times_dV_d^H,
\end{equation*}
where the matrices $A_i^{(k_i)}$ have of small/medium size. 

We remark that in this case it is not guaranteed that the tensor ${\mathcal C}_{\vec k}$ can be fully stored since its size grows exponentially with $d$. A way to overcome this issue is to use an algorithm for the solution of the projected tensor Sylvester equation that keeps the solution in TT format, such as the AMEn algorithm described in \cite{dolgov2014alternating}.

 \section{Pole selection}\label{sec:pole_selection}

In this section we derive techniques for pole selection, employing a representation of the residual that involves the poles of the block rational Krylov subspaces. 

Thanks to Corollary~\ref{cor:residual-norm} and Remark~\ref{eqn:residual_unfolding}, the analysis can be reduced to the problem of minimizing the norms of the residuals $(\mathcal{R}_i)_{(i)}$ of the Sylvester equations $A_iX - XB_i = (C_i)_{(i)}$, solved projecting $A_i$ into the block rational Krylov subspace $\rat_{k_i}(A, C_i)$. Since this work is devoted to studying the case of $\mathcal{C}$ in Tucker or TT format, we also assume that the matrices $(C_i)_{(i)}$ admit a low rank representation $(C_i)_{(i)}=C_iZ_i^H$, as it has been shown in Sections~\ref{sec:Tucker} and \ref{sec:TT} for the case of $\mathcal C$ in TT or Tucker format. 
 
 To represent the norm of the matrices $(\mathcal{R}_i)_{(i)}$ with a formulation that involves the chosen poles we can use a simplified version of Theorem~6.1 of \cite{casulli2022}.

 \begin{theorem} \label{thm:block-residual}
    Let $A \in \C^{n\times n}$, $B\in \C^{m\times m}$, $C\in \C^{n\times b}$ and $Z\in \C^{m\times b}$. Let $V\in \C^{n\times b\cdot h}$ be a matrix with orthonormal columns that spans $\rat_h(A,C,\vec {\xi})$ and let $A_h=V^HAV$. Let $X_{h}=V^HY_{h}$ where $Y_{h}$ is the solution of the Sylvester equation
    \begin{equation*}
        A_hY_{h}-Y_{h}B=C^{(h)} Z^H, 
    \end{equation*}
    where $C^{(h)}=V^HC.$ 
    Let $\chi_A(z)\in \pol_h(\C^{b\times b})$ be the monic block characteristic polynomial of $A_h$ with respect to $C^{(h)}$. Define $R_A^G(z)=\frac{\chi_A(z)}{Q_A(z)},$ where 
    \begin{equation*}
        Q_A(z)=\prod_{\xi\in\vec \xi, \xi \neq \infty}(z-\xi).
    \end{equation*}

  Then the residual matrix is equal to 
  \begin{equation*}
      \left(R_A^G(A)\circ C\right)({R_A^G}^H(B)\circ^{-1} Z)^H.
    \end{equation*}
\end{theorem}
\begin{proof}
    It is sufficient to consider Theorem~6.1 in \cite{casulli2022} with $V=I_m$.
\end{proof}

From now on we assume that for each $i$, the first pole in $\vec \xi_i$ is equal to infinity, that is, the first block column of $V_i$ is an orthonormal basis of the space spanned by the columns of $C_i$. The first consequence of this hypothesis is that the term $\hat{c}$ in the formulation of the residual vanishes, hence the convergence of the residual can be monitored just by the Frobenius norms of the matrices  $(\mathcal{R}_i)_{(i)}$. Notice that, thanks to Theorem~\refeq{thm:block-residual}, we have 
\begin{equation}\label{eqn:res-sylv}
    (\mathcal R_i)_{(i)}=\left( R_{i}^G(A_i)\circ C_i\right)({ R_{i}^G}^H(-B_i)\circ^{-1}  Z_i)^H, \quad \text{ with }  \quad  R_{i}^G(z)=\chi_{i}(z)/Q_{i}(z),
\end{equation}
where $\chi_{i}(z)$ is the monic block characteristic polynomial of ${A}_i^{(k_i)}$ with respect to $V_i^HC_i$ and 
\begin{equation*}
    Q_i(z)=\prod_{\xi\in\vec \xi_i, \xi \neq \infty}(z-\xi).
\end{equation*}

As it has been established in \cite{casulli2022}, to keep the Frobenius norm of \eqref{eqn:res-sylv} small it is sufficient to choose poles that minimize the norm of $R_{i}^G(z)^{-1}$ for every $z$ in the field of values of $-B_i$. A way to approximatively minimize such norm is to adaptively choose the next pole for $\vec \xi_i$ accordingly with one of the following two possibilities:
\begin{enumerate}
    \item the first method, denoted by {\texttt{det}}, is to choose the new pole as the conjugate of
    \begin{equation}\label{eqn:det}
            \arg\max_{\lambda\in \W(-B_i)} \frac{\prod_{\xi\in \vec{\xi}_i, \xi\neq \infty}|\lambda-\bar{\xi}|^b}{\prod_{\mu\in\Lambda\left({A}^{(k_i)}_i\right)}|\lambda-\bar{\mu}|},
    \end{equation}
    \item to introduce the second method, denoted by {\texttt{det2}}, for each $\lambda \in \W(-B_i)$, let $\mu_1,\dots, \mu_{bk_i}$ be the eigenvalues of $A^{(k_i)}_i$, ordered such that $|\bar{\lambda}-\mu_1|\le|\bar{\lambda}-\mu_2|\le\dots\le|\bar{\lambda}-\mu_{bk_i}|$. The new pole is chosen as the conjugate of
    \begin{equation}\label{eqn:det2}
        \arg \max_{\lambda \in \W(-B_i)}\frac{\prod_{\xi\in\vec{\xi}_B, \xi\neq \infty}(\lambda-\bar{\xi})}{\prod_{j=1}^{k-1}|\lambda-\bar{\mu}_{(j-1)b+1}|}.
    \end{equation}
\end{enumerate}

The main advantage of choosing poles accordingly with {\tt det2} instead of {\tt det} is that we have to minimize a rational function with a much smaller degree. Moreover, from the numerical experiments made in \cite{casulli2022} it appears that in the case of $d=2$, {\tt det2} has comparable or better performances than {\tt det}.
\begin{remark}\label{rmk:approx_WA}
    The field of values of the matrix $-B_i$ is the sum of the field of values of the matrices $-{A}_j^{(k_j)}$, for $j\neq i$. In general, the determination of the field of values of $-A_j^{(k_j)}$ is not easy. What we do in practice is to substitute the field of values with the convex hull of the set obtained by taking the union of the eigenvalues of the matrices $-A_j^{(s)}$ for $s\le k_j$.
\end{remark}

\begin{remark}\label{rmk:boundary_W(-B)}
    If $\W(-B_i)$ has a nonempty interior, for the maximum modulus principle it is sufficient to maximize the functions \eqref{eqn:det} and \eqref{eqn:det2} over its boundary.
\end{remark}

\section{Computation of the residual}\label{sec:residual}
The explicit computation of the residual to monitor the convergence of the algorithm is usually expensive; to overcome this problem we can compute the norms of the matrices $(\mathcal R_i)_{(i)}$ and then recover the norm of the residual using the result of Corollary~\refeq{cor:residual-norm}. If the last pole is equal to infinity, the norms of the partial residuals can be cheaply computed thanks to the following lemma.

\begin{lemma}
Let $A\in \C^{n\times n}$, $B\in \C^{m\times m}$, $C\in\C^{n\times b}$, $Z \in \C^{m\times b}$, let $V_{h+1}$ be a block orthonormal basis of $\rat_h(A,C, \xi)$, and let $\underline{H}_{h},{K}_{h}$ be matrices generated by the block rational Arnoldi algorithm, where $\xi_0=\xi_h=\infty$. Denote by $V_h$ the matrix obtained by removing from $V_{h+1}$ the last $b$ columns and denote by $\rho_h$ the residual
\begin{equation}\label{eqn:large-residual}
    \rho_h=AV_{h}Y + V_{h} Y B - C Z^H,
\end{equation}
where $Y$ solves 
\begin{equation}\label{eqn:small-residual}
    (V_h^HAV_{h})Y + Y B - V_h^HC Z^H=0.
\end{equation}
Then  \begin{equation*}
    \rho_h= V_{h+1}\vec e_{h+1}\underline{H}_hK_h^{-1}Y \quad \text{ and }\quad \norm{\rho_h}_F =\norm{\vec e_{h+1}\underline{H}_hK_h^{-1}Y}_F,
\end{equation*}
where $\vec e_{h+1}=[0, I_b]$, and ${K}_{h}$ is the leading principal $bh\times bh$ submatrix of $\underline{K}_k$.
\end{lemma}

\begin{proof}
    From \eqref{eqn:small-residual} follows that 
    \begin{equation*}
        Y B=-(V_h^HAV_{h})Y +V_h^HC Z^H,
    \end{equation*}
    hence \eqref{eqn:large-residual} can be rewritten as
\begin{equation*}
    R=AV_{h}Y - V_{h}(V_h^HAV_{h})Y 
    +V_hV_h^HC Z^H- C Z^H.
\end{equation*}
    Since $\xi_0=\infty$, we have $C\in \rat_h(A,C, \vec \xi)$, then $V_hV_h^HC = C$. In particular we have
    \begin{equation*}
        R=AV_{h}Y - V_{h}(V_h^HAV_{h})Y.
    \end{equation*}
    Note now that since $\xi_h=\infty$, $V_{h+1}V_{h+1}^HAV_{h}=AV_h$, hence 
    \begin{equation}
        R=V_{h+1}\left(V_{h+1}^HA-\begin{bmatrix}
            V_h^H\\0
        \end{bmatrix}A\right)V_{h}Y=V_{h+1}\begin{bmatrix}
            0\\ e_{h+1}^HV_{h+1}^H
        \end{bmatrix}AV_h Y,
    \end{equation}
where $ e_{h+1}=[0, I_b]$. Moreover, we have that $AV_h=V_{h+1}\underline{H}_hK_h^{-1}$ (see \cite[Section~4]{casulli2022} for more details) and since the columns of $V_{h+1}$ are orthonormal we have 
\begin{equation*}
    R= V_{h+1} e_{h+1}\underline{H}_hK_h^{-1}Y \quad \text{ and }\quad \norm{R}_F =\norm{ e_{h+1}\underline{H}_hK_h^{-1}Y}_F.
\end{equation*}

\end{proof}

To avoid the multiplication by the (possibly large) matrices $A_i$, we can use block rational Krylov methods that start with a pole equal to infinity and after each step swap the last two poles guaranteeing that the last pole is always equal to infinity, as illustrated in \cite[Section~5]{casulli2022}. 

Summarizing, if we perform a \tbrk method where, for each block rational Krylov subspace, we have $\xi_0=\infty$ and we guarantee that the last pole is equal to infinity, we can write the norm of the residual as 
\begin{equation*}
    \norm{r(\vec{V},A_1,\dots,A_d,c)}_2
    =\sqrt{\sum_{i=1}^d \norm{\mathcal{Y}_{\vec k}\times_ie_{k_{i}+1}^{H}\underline{H}^{(i)}_{k_i}(K^{(i)}_{k_i})^{-1}}_F^2}.
\end{equation*}
where $\mathcal{Y}_{\vec k}$ is the solution of the projected equation, for each $i$ we have $e_{k_i+1}=[0,I_b]^H\in \C^{b(k_i+1)\times b}$, and $\underline{H}^{(i)}_{k_i},K^{(i)}_{k_i}$ are generated by the block rational Arnoldi algorithm for the computation of $\rat_{k_i}(A,C_i)$.
    
\section{Numerical results} \label{sec:num-exp}
In this section, we provide numerical results on the convergence of the presented algorithms, for the solution of tensor Sylvester equations with right hand side represented in Tucker or Tensor Train format. The MATLAB code of the algorithms used for solving tensor Sylvester equations has been made freely available at \url{https://github.com/numpi/TBRK-Sylvester}.

As a first test problem, we compute the approximate solution of the Poisson equation on a $d$-dimensional hypercube  
\begin{equation*}
    \begin{cases}
        -\Delta u = f &\text{ in }\Omega\\
        u\equiv 0 &\text{ on }\partial\Omega
    \end{cases}, \qquad \Omega= [0, 1]^d.
\end{equation*}

Unless otherwise specified we discretize the domain with a uniformly spaced grid with $n=1024$ points in each direction, and the operator $\Delta$ by centered finite differences, which yields the tensor Sylvester equation

\begin{equation*}
    \mathcal{X}\times_1A+\dots+ \mathcal{X}\times_d A=\mathcal{F}, \quad \text{ with } \quad
    A=\frac{1}{h^2}\begin{bmatrix}
        2&-1\\
        -1&2&\ddots\\
        &\ddots&\ddots&-1\\
        &&-1&2
    \end{bmatrix},
\end{equation*}
where $h=\frac{1}{n-1}$ is the distance between the grid points and $\mathcal{F}$ is the tensor given by sampling $f$ on the grid points. If the function $f$ is a smooth multivariate function, the tensor $\mathcal{F}$ is numerically low-rank, that is it can be approximated by a low multilinear or TT rank tensor, see \cite{shi2021compressibility}. 

To test the algorithms also in the case of non symmetric $A_i$s, we consider as a second test problem the approximate solution of the convection-diffusion partial differential equation
\begin{equation*}
    \begin{cases}
        -\epsilon\Delta u + w \cdot \nabla u = f &\text{ in }\Omega\\
        u\equiv 0 &\text{ on }\partial\Omega
    \end{cases}, \qquad \Omega= [0, 1]^d,
\end{equation*}
where $\epsilon>0$ is the viscosity parameter and $ w$ is the convection vector. As described in \cite{palitta2016matrix}, assuming $ w=(\Phi_1(x_1),\Phi_2(x_2),\dots,\Phi_d(x_d))$, and discretizing the domain with a uniformly spaced grid as before, we obtain the tensor Sylvester equation

\begin{equation*}
    \mathcal{X}\times_1(\epsilon A+\vec{\Phi}_1 B)+\dots+ \mathcal{X}
    \times_d (\epsilon A+\vec{\Phi}_d B )=\mathcal{F}, 
\end{equation*}
where $A$ and $\mathcal{F}$  are defined as in the previous test problem,
\begin{equation*}
    \vec {\Phi}_i=\begin{bmatrix}
        \Phi_i(h)\\
        &\Phi_i(2h)\\
        &&\ddots\\
        &&&\Phi_i((n-2)h)
    \end{bmatrix},
\end{equation*}
 for $i=1,\dots,d$ and 
\begin{equation*}
    B=\frac{1}{2h}\begin{bmatrix}
        0&1\\
        -1&\ddots&\ddots\\
        &\ddots&\ddots&1\\
        &&-1&0
    \end{bmatrix}
\end{equation*}
is the discretization by centered finite differences of the first order derivative in each direction.

The numerical simulations have been run on a server with two
Intel(R) Xeon(R) E5-2650v4 CPU running at 2.20 GHz and 256 GB of RAM, using MATLAB R2021a with the Intel(R) Math Kernel Library Version 2019.0.3. All the experiments are made in double precision, real arithmetic. In particular, if a nonreal pole is employed during a Krylov method,  the subsequent is chosen as its conjugate. This allows us to keep the matrices and the tensors real. We refer the reader to \cite{ruhe1994rational2} for a more complete discussion.

During the experiments, different choices of poles are used. In particular, we denote by {\tt det} and {\tt det2} the poles described in \eqref{eqn:det} and \eqref{eqn:det2}, respectively, we indicate by {\tt poly} the use of all poles equal to infinity any by {\tt ext} the case in which the poles are chosen alternating $0$ and infinity.

For the computation of rational Krylov subspaces we employ the \texttt{rktoolbox} described in \cite{berljafa2014rational}. The maximization problems that appear in \eqref{eqn:det} and \eqref{eqn:det2} are solved by maximizing the functions on a sampling of the boundary of the set described in Remark~\ref{rmk:approx_WA}.

\subsection{Tucker format}
In this section, we provide numerical results for the case of right hand side in Tucker format employing the algorithm described in Section~\ref{sec:Tucker}, denoted by \tuckTBRK.

In Figure~\ref{plot:diffusion-d4} we show the behavior of the relative norm of the residual by varying the number of Arnoldi iterations, for the solution of a $4$-dimensonal discretized Poisson equation using different choices of poles.
In Figure~\ref{plot:diffusion-d3-comparing-n} the behavior of the relative norm of the residual is compared for different sizes of the discretization grid for solving a $3$-dimensional Poisson equation.

In the case of discretized convection-diffusion equation, it is not guaranteed that the matrices $A_i$ are symmetric, hence complex poles could appear. As noticed at the beginning of the section, the complex poles are employed coupled with their conjugates to guarantee real arithmetic. For this reason, the number of Arnoldi iterations performed in the construction of different Krylov subspaces $\rat_{k_i}(A_i,C_i)$ may be different. 

Figure~\ref{plot:convection-d3} shows the behavior of the relative norm of the residual by varying the mean number of Arnoldi iterations, for the solution of a $3$-dimensional discretized convection-diffusion equation, where $\epsilon=0.1$ and $ w=(1 + \frac{(x_1 + 1)^2}{4},0,0)$ using different choices of poles. Table~\ref{table:tuck} shows the number of Arnoldi iterations needed to reach a relative norm of the residual less than $10^{-4}$ and $10^{-6}$. Moreover, the table also contains the time of execution of the algorithms.

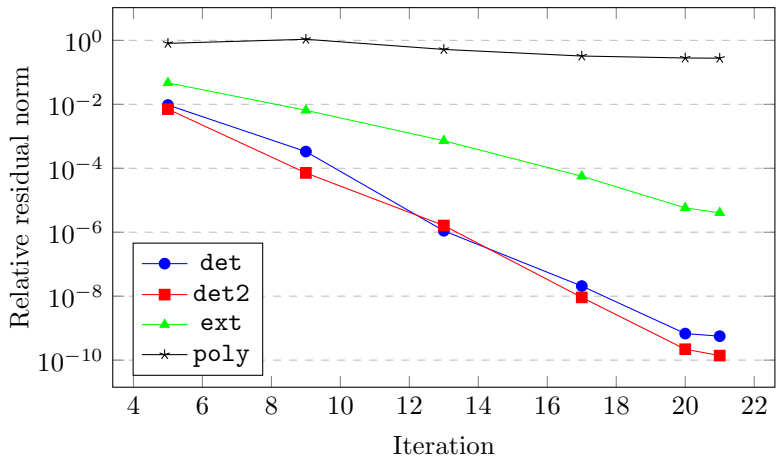
\begin{figure}
	\makebox[\linewidth][c]{
		\begin{tikzpicture}
			\begin{semilogyaxis}[
				title = {},  
				xlabel = {Iteration},
				ylabel = {Relative residual norm},
				x tick label style={/pgf/number format/.cd,%
					scaled x ticks = false,
					set thousands separator={},
					fixed},
				legend pos=south west,
				ymajorgrids=true,
                height=.45\textwidth,
				grid style=dashed]
		
				\addplot[color=blue, mark=*, mark size=2pt] table {Tuck_symm_d4_det.dat};
				\addlegendentry{{\tt{det}}}

				\addplot[color=red, mark=square*, mark size=2pt] table {Tuck_symm_d4_det2.dat};
				\addlegendentry{{\tt{det2}}}

                \addplot[color=green, mark=triangle*, mark size=2pt] table {Tuck_symm_d4_ext.dat};
				\addlegendentry{{\tt{ext}}}

                \addplot[color=black, mark=star, mark size=2pt] table {Tuck_symm_d4_poly.dat};
				\addlegendentry{{\tt{poly}}}
			\end{semilogyaxis}
	\end{tikzpicture}}
\caption{Behavior of the relative norm of the residual produced by solving the discretized Poisson equation of dimension $d=4$ with $f=1/((1+x_1+x_2)(1+x_3+x_4))$ employing \tuckTBRK methods, with different choices of poles.} \label{plot:diffusion-d4}
\end{figure}

\begin{figure}
	\makebox[\linewidth][c]{
		\begin{tikzpicture}
			\begin{semilogyaxis}[
				title = {},  
				xlabel = {Iteration},
				ylabel = {Relative residual norm},
				x tick label style={/pgf/number format/.cd,%
					scaled x ticks = false,
					set thousands separator={},
					fixed},
				legend pos=south west,
				ymajorgrids=true,
                height=.45\textwidth,
				grid style=dashed]
		
				\addplot[color=blue, mark=*, mark size=2pt] table {Tuck_symm_d3_n=128.dat};
				\addlegendentry{{$n=128$}}

				\addplot[color=red, mark=square*, mark size=2pt] table {Tuck_symm_d3_n=256.dat};
				\addlegendentry{{$n=256$}}

                \addplot[color=green, mark=triangle*, mark size=2pt] table {Tuck_symm_d3_n=512.dat};
				\addlegendentry{{$n=512$}}

                \addplot[color=black, mark=star, mark size=2pt] table {Tuck_symm_d3_n=1024.dat};
				\addlegendentry{{$n=1024$}}
			\end{semilogyaxis}
	\end{tikzpicture}}
\caption{Behavior of the relative norm of the residual produced by solving the discretized Poisson equation of dimension $d=3$ with $f=1/(1+x_1+x_2+x_3)$ employing \tuckTBRK methods, with poles chosen accordingly to \texttt{det} for different sizes of the discretization grid.} \label{plot:diffusion-d3-comparing-n}
\end{figure}
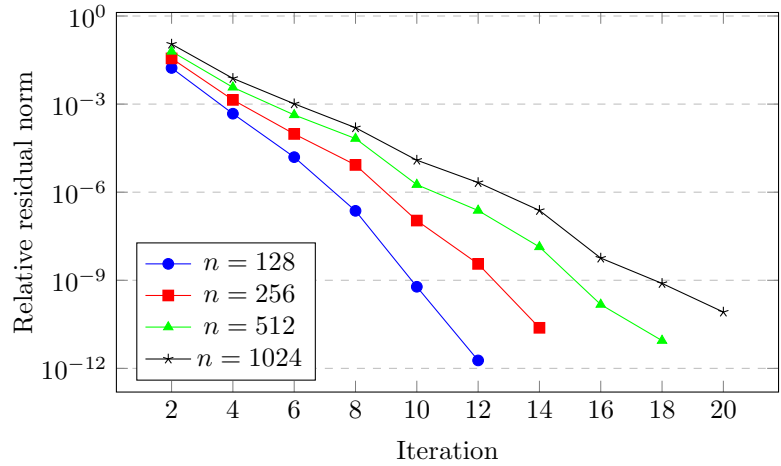

\begin{figure}
	\makebox[\linewidth][c]{
		\begin{tikzpicture}
			\begin{semilogyaxis}[
				title = {},  
				xlabel = {Mean of iterations},
				ylabel = {Relative residual norm},
				x tick label style={/pgf/number format/.cd,%
					scaled x ticks = false,
					set thousands separator={},
					fixed},
				legend pos=south west,
				ymajorgrids=true,
                height=.45\textwidth,
				grid style=dashed]
		
				\addplot[color=blue, mark=*, mark size=2pt] table {Tuck_nonsymm_d3_det.dat};
				\addlegendentry{{\tt{det}}}

				\addplot[color=red, mark=square*, mark size=2pt] table {Tuck_nonsymm_d3_det2.dat};
				\addlegendentry{{\tt{det2}}}

                \addplot[color=green, mark=triangle*, mark size=2pt] table {Tuck_nonsymm_d3_ext.dat};
				\addlegendentry{{\tt{ext}}}

                \addplot[color=black, mark=star, mark size=2pt] table {Tuck_nonsymm_d3_poly.dat};
				\addlegendentry{{\tt{poly}}}
			\end{semilogyaxis}
	\end{tikzpicture}}
\caption{Behavior of the relative norm of the residual produced by solving a discretized convection-diffusion equation of dimension $d=3$ with $f=1/((1+x_1+x_2+x_3))$ employing \tuckTBRK methods, with different choices of poles.}\label{plot:convection-d3}
\end{figure}
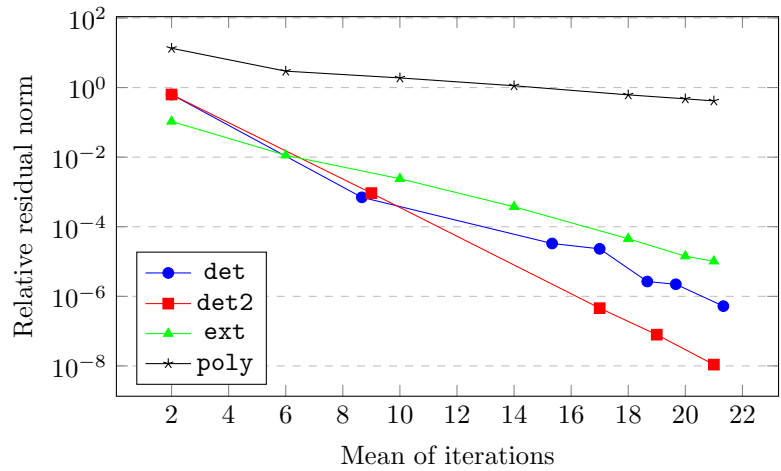

\begin{table}
\centering
    \begin{tabular}{||c||ccc||ccc||}
        \hline
        poles & iterations & time (s) & residual & iterations & time (s) & residual \\
        \hline\hline
        det & $ 11$  $20$  $20 $ & $6.06 $&$ 2.32e-05 $ & $ 15$  $27$  $27 $ & $15.66 $&$ 2.25e-07 $\\ 
        det2 & $9$  $12$  $12 $ &$ 2.34 $&$ 8.83e-05$ & $17$  $20 $ $20 $ &$ 8.75 $&$ 7.91e-08$ \\
        ext & $ 19$  $19 $ $19 $ &$ 10.36 $&$ 3.35e-05 $ &$ 25$  $25$  $25 $ &$ 28.70 $&$ 5.17e-07 $ \\
        \hline
    \end{tabular}
    \caption{Iterations and time needed to reach a relative norm of the residual less than $10^{-4}$ (left) and $10^{-6}$ (right) for the solution of discretized convection-diffusion equation of dimension $d=3$, with $f=1/((1+x_1+x_2+x_3))$ employing \tuckTBRK methods, with different choices of poles.}\label{table:tuck}
\end{table}

\subsection{Tensor Train format}
In this section, we provide numerical results for the case of right hand side in Tensor Train format, employing the algorithm described in Section~\ref{sec:TT}, denoted by \ttTBRK.
We have implemented the 
\ttTBRK methods in MATLAB, using the TT-Toolbox \cite{oseledetstt} to manage tensors in TT format. 

In Figure~\ref{plot:TT-diffusion-d6} we show the behavior of the relative norm of the residual by varying the number of Arnoldi iterations, for the solution of a $6$-dimensonal discretized Poisson equation employing \ttTBRK methods with different choices of poles. 
In Figure~\ref{plot:TT-diffusion-d5-comparing-n} the behavior of the relative norm of the residual is compared for different sizes of the discretization grid for solving a $5$-dimensional Poisson equation.
Figure~\ref{plot:TT-convection-d5} shows the behavior of the relative norm of the residual by varying the mean number of Arnoldi iterations, for the solution of a $5$-dimensonal discretized convection-diffusion equation, where $\epsilon=0.1$ and $ w=(1 + \frac{(x_1 + 1)^2}{4},\frac{(1+x_2)}{2},0,0,0)$ employing \ttTBRK methods with different choices of poles.

In Table~\ref{table:Amen-comparation} we compare the execution time of \ttTBRK and AMEn, to reach a relative norm of the residual less than $10^{-8}$ for the solution of a $d$ dimensional Poisson equation for different values of $d$. We remark that in the first two cases AMEn does not reach the required accuracy. 

To show the potentiality of the presented algorithm for the solution of high dimensional PDEs, we report in Table~\ref{table:tt-tbrk-large-d} the time and the number of Arnoldi iterations employed by \ttTBRK for the computation of the solution of high dimensional Poisson equations with a relative norm of the residual less than $10^{-6}$. From the results it appears that the number of iterations does not grow up when the space dimension $d$ increases. The more than linear increase of the computational time is due to the resolution of the small-size tensor Sylvester equation by the AMEn algorithm.

\begin{figure}
	\makebox[\linewidth][c]{
		\begin{tikzpicture}
			\begin{semilogyaxis}[
				title = {},  
				xlabel = {Iteration},
				ylabel = {Relative residual norm},
				x tick label style={/pgf/number format/.cd,%
					scaled x ticks = false,
					set thousands separator={},
					fixed},
				legend pos=south west,
				ymajorgrids=true,
                height=.45\textwidth,
				grid style=dashed]
		
				\addplot[color=blue, mark=*, mark size=2pt] table {TT_symm_d6_det.dat};
				\addlegendentry{{\tt{det}}}

				\addplot[color=red, mark=square*, mark size=2pt] table {TT_symm_d6_det2.dat};
				\addlegendentry{{\tt{det2}}}

                 \addplot[color=green, mark=triangle*, mark size=2pt] table {TT_symm_d6_ext.dat};
				 \addlegendentry{{\tt{ext}}}

                 \addplot[color=black, mark=star, mark size=2pt] table {TT_symm_d6_poly.dat};
				 \addlegendentry{{\tt{poly}}}
			\end{semilogyaxis}
	\end{tikzpicture}}
\caption{Behavior of the relative norm of the residual produced by solving the discretized Poisson equation of dimension $d=6$ with random right hand side of TT rank $(2,2,\dots,2)$, employing \ttTBRK methods, with different choices of poles.} \label{plot:TT-diffusion-d6}
\end{figure}
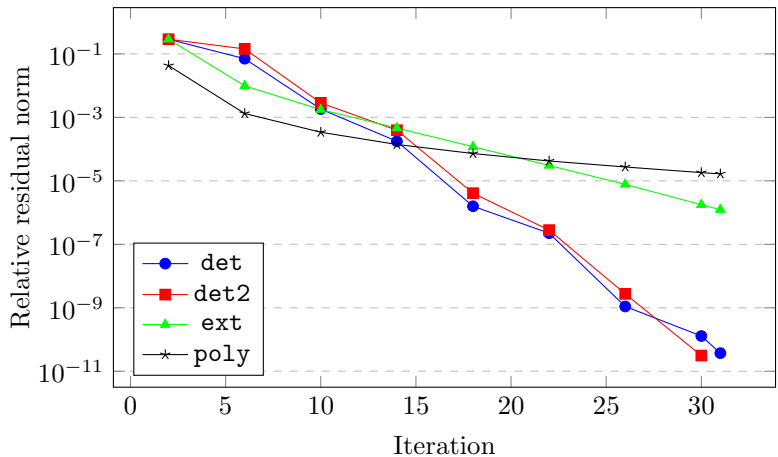

\begin{figure}
	\makebox[\linewidth][c]{
		\begin{tikzpicture}
			\begin{semilogyaxis}[
				title = {},  
				xlabel = {Iteration},
				ylabel = {Relative residual norm},
				x tick label style={/pgf/number format/.cd,%
					scaled x ticks = false,
					set thousands separator={},
					fixed},
				legend pos=south west,
				ymajorgrids=true,
                height=.45\textwidth,
				grid style=dashed]
		
				\addplot[color=blue, mark=*, mark size=2pt] table {TT_symm_d5_n=128.dat};
				\addlegendentry{{$n=128$}}

				\addplot[color=red, mark=square*, mark size=2pt] table {TT_symm_d5_n=256.dat};
				\addlegendentry{{$n=256$}}

                 \addplot[color=green, mark=triangle*, mark size=2pt] table {TT_symm_d5_n=512.dat};
				 \addlegendentry{{$n=512$}}

                 \addplot[color=black, mark=star, mark size=2pt] table {TT_symm_d5_n=1024.dat};
				 \addlegendentry{{$n=1024$}}
			\end{semilogyaxis}
	\end{tikzpicture}}
\caption{Behavior of the relative norm of the residual produced by solving the discretized Poisson equation of dimension $d=5$ with random right hand side of TT rank $(2,2,\dots,2)$, employing \ttTBRK methods,  with poles chosen accordingly to \texttt{det} for different sizes of the discretization grid.} \label{plot:TT-diffusion-d5-comparing-n}
\end{figure}
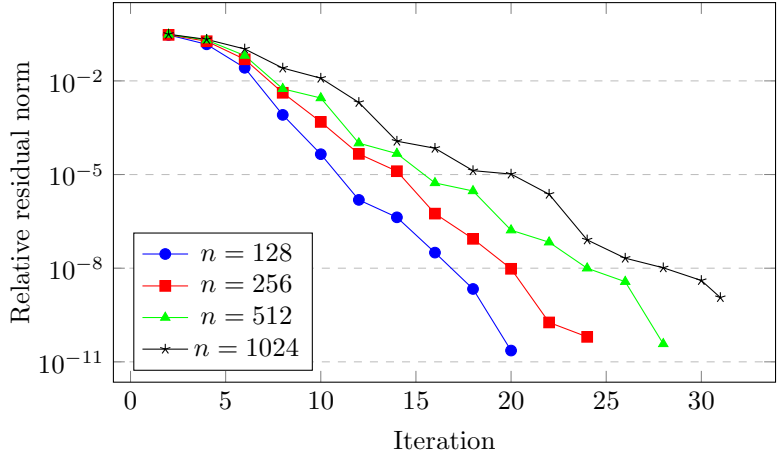

    \begin{figure}
        \makebox[\linewidth][c]{
            \begin{tikzpicture}
                \begin{semilogyaxis}[
                    title = {},  
                    xlabel = {Mean of iterations},
                    ylabel = {Relative residual norm},
                    x tick label style={/pgf/number format/.cd,%
                        scaled x ticks = false,
                        set thousands separator={},
                        fixed},
                    legend pos=south west,
                    ymajorgrids=true,
                    height=.45\textwidth,
                    grid style=dashed]
            
                    \addplot[color=blue, mark=*, mark size=2pt] table {TT_nonsymm_d5_det.dat};
                    \addlegendentry{{\tt{det}}}
    
                    \addplot[color=red, mark=square*, mark size=2pt] table {TT_nonsymm_d5_det2.dat};
                    \addlegendentry{{\tt{det2}}}
    
                    \addplot[color=green, mark=triangle*, mark size=2pt] table {TT_nonsymm_d5_ext.dat};
                    \addlegendentry{{\tt{ext}}}
    
                    \addplot[color=black, mark=star, mark size=2pt] table {TT_nonsymm_d5_poly.dat};
                    \addlegendentry{{\tt{poly}}}
                \end{semilogyaxis}
        \end{tikzpicture}}
    \caption{Behavior of the relative norm of the residual produced by solving a discretized convection-diffusion equation of dimension $d=5$ with random right hand side of TT rank $(2,2,\dots,2)$, employing \ttTBRK methods, with different choices of poles.}\label{plot:TT-convection-d5}
    \end{figure}
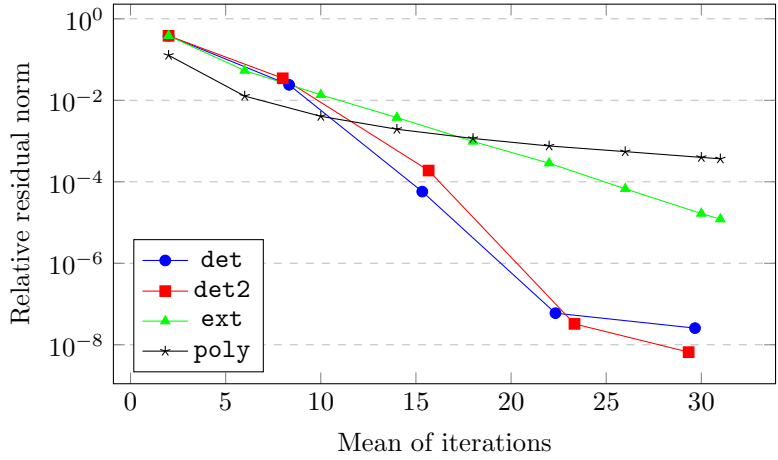 
    
    \begin{table}
        \centering
            \begin{tabular}{|c|ccc|}
                \hline
                & d & residual & time (s) \\
                \hline \hline
\ttTBRK & $3$ & $6.00e-09$ & $17.63$ \\
AMEn & $3$ & $2.41e-07$ & $2044.36$ \\
\hline \hline
\ttTBRK & $4$ & $6.91e-09$ & $56.10$ \\
AMEn & $4$ & $2.69e-08$ & $6671.18$ \\
\hline \hline
\ttTBRK & $5$ & $7.53e-09$ & $78.70$ \\
AMEn & $5$ & $6.25e-09$ & $2918.36$ \\
    \hline
     \end{tabular}
            \caption{Comparison of execution time and accuracy between \ttTBRK with poles chosen accordingly with {\tt det2} and AMEn to reach relative norm of the residual less than $10^{-8}$ for the solution of a $d$ dimensional Poisson equation for different values of $d$.} \label{table:Amen-comparation}
        \end{table}

        \begin{table}
            \centering
                \begin{tabular}{|cccc|}
                    \hline
 $d$ & residual  & Arnoldi iterations &  time (s) \\
 \hline \hline
$5$ & $5.30e-08$ & $26$ & $13.80$ \\
$10$ & $7.07e-07$ & $26$ & $20.28$ \\
$15$ & $7.85e-07$ & $26$ & $163.40$ \\
$20$ & $1.41e-07$ & $22$ & $4507.42$ \\
\hline
         \end{tabular}
                \caption{Time, accuracy and number of Arnoldi iterations of \ttTBRK with poles chosen accordingly with {\tt det} required to reach relative norm of the residual less than $10^{-6}$ for the solution of a $d$ dimensional Poisson equation for large values of $d$. } \label{table:tt-tbrk-large-d}
            \end{table}

\section{Conclusions} In this work we have provided a characterization of tensorized block rational Krylov subspaces using multivariate rational functions. We have also developed a method for solving tensor Sylvester equations with low multilinear or Tensor Train rank, based on Galerkin projection onto a tensorized block rational Krylov subspace, providing a convergence analysis. Generalizing the results of \cite{casulli2022}, we have developed strategies for pole selection and efficient techniques for the computation of the residual based on poles reordering. 
We expect that tensorized block rational Krylov subspaces can be used for solving more general high dimensional tensor problems, such as the computation of functions of matrices with multiterm Kronecker structures.

The code of the resulting algorithm for solving tensor Sylvester equations has been made freely available at \url{https://github.com/numpi/TBRK-Sylvester}.

\section*{Acknowledgements} The author would like to thank Michele Benzi and Leonardo Robol for their support and advice.

    \bibliographystyle{plain}
    \bibliography{biblio_tensor}

\end{document}